\documentclass[12pt]{amsart}
\usepackage{amsmath,amssymb}

\theoremstyle{plain}
\newtheorem{thm}{Theorem}[section]
\newtheorem{theorem}[thm]{Theorem}

\newtheorem{lemma}[thm]{Lemma}
\newtheorem{corollary}[thm]{Corollary}
\newtheorem{proposition}[thm]{Proposition}
\theoremstyle{definition}
\newtheorem{remark}[thm]{Remark}

\newtheorem{examples}[thm]{Examples}

\numberwithin{equation}{section}

\newcommand{\sC}{{\mathcal C}}

\newcommand{\sK}{{\mathcal K}}

\newcommand{\sU}{{\mathcal U}}

\newcommand{\C}{{\mathbb C}}

\newcommand{\G}{{\mathbb G}}

\newcommand{\BP}{{\mathbb P}}
\newcommand{\Q}{{\mathbb Q}}

\newcommand{\BS}{{\mathbb S}}

\newcommand{\pit}{{\mathbb P}}
\newcommand{\qit}{{\mathbb Q}}
\newcommand{\0}{{\mathcal O}}

\newcommand{\Lag}{{\rm Lag}}
\newcommand{\SYM}{{\rm Sym}}

\newcommand{\Gr}{{\rm Gr}}

\title[Equivariant compactifications of vector groups with high index]{Equivariant compactifications of vector groups with high index}
\author{Baohua Fu and Pedro Montero}
\begin{document}
\maketitle  \setcounter{tocdepth}{1}

\begin{abstract}
In this note, we classify smooth equivariant compactifications of $\mathbb{G}_a^n$ which are Fano manifolds with index $\geq n-2$.
\end{abstract}

\section{Introduction}

 Let $\G^n_a$ be the complex vector group of dimension $n$, i.e., $\C^n$
equipped with the additive group structure.
 A {\em smooth equivariant
compactification}
 of $\G_a^n$ (SEC in abbreviation) is    a $\G_a^n$-action
 $A: \G_a^n \times X \to X$ on a projective manifold $X$ of dimension $n$
  with an open orbit $O \subset X$.  Since $\G_a^n$ contains no nontrivial finite subgroup,
  the open orbit $O$ is isomorphic to $\C^n$. When
 our interest is on the underlying manifold $X$, we say that $X$ {\em is a SEC}.

The study of SEC is started from \cite{HT}, where a classification in dimension 3 and of Picard number one is obtained.
Recently a full classification of Fano 3-folds which are SEC is obtained in \cite{HM}, while it seems difficult to pursue further in higher dimension.
In \cite{FH3}, the first author and J.-M. Hwang introduced the notion of Euler-symmetric projective varieties, namely nondegenerate projective varieties admitting many $\C^*$-actions of Euler type. It is shown in {\em loc. cit.} that they are equivariant compactifications of vector groups and they are classified by certain algebraic data (called symbol systems), while it remains the problem to translate the smoothness in terms of these algebraic data.

Recall that for a Fano manifold $X$ of dimension $n$, its {\em index} $i_X$ is the greatest integer such that $-K_X = i_X H$ for some divisor $H$ on $X$.
It is well-known that $i_X \leq n+1$. By a series of works of Fujita (\cite{F1}, \cite{F2}, \cite{F3}), Mella (\cite{Me}),  Mukai (\cite{M}) and Wi\'sniewski (\cite{W1}),  the classification of Fano $n$-folds with index $i_X \geq n-2$ is known. Based on this, we will give a classification of SEC $n$-folds with index $\geq n-2$. In the case of a prime Fano variety, that is, of a Fano variety with Picard number one, our result reads:

\begin{theorem} \label{t.main}
Let $X$ be an $n$-dimensional SEC with Picard number one. Assume that $i_X \geq n-2$. Then $X$ is isomorphic to one of the following:

\begin{itemize}
\item[(1)] 6 homogeneous varieties of algebraic groups:  $\BP^n, \Q^n, {\rm Gr}(2,5), {\rm Gr}(2,6), \BS_5, {\rm Lag}(6)$.
\item[(2)] 5 non-homogeneous varieties:
\begin{itemize}
\item[(2-a)] smooth linear sections of ${\rm Gr}(2,5)$ of codimension 1 or 2.
\item[(2-b)] $\BP^4$-general\footnote{See \cite[Definition 2.5]{FH2}.} linear sections of $\BS_5$ of codimension 1,2 or 3.
\end{itemize}
\end{itemize}
\end{theorem}

Here $\BS_5$ and ${\rm Lag}(6)$ denote the 10-dimensional spinor variety and the 6-dimensional Lagrangian Grassmannian, respectively.
As a corollary of Theorem \ref{t.main}, we obtain the following
\begin{corollary} \label{c.class}
Let $X$ be a SEC $n$-fold of Picard number one. Assume that the VMRT at a general point of $X$ is smooth (e.g. $X$ is covered by lines) and $n \leq 5$. Then
$X$ is isomorphic to one of the following:
$$
\BP^n, \Q^n,\ \text{smooth linear sections of codimension 1 or 2 of }\ {\rm Gr}(2,5).
$$
\end{corollary}

It is expected that for a SEC of Picard number one, its VMRT at a general point is always smooth (\cite{FH3}).

To complete the classification of smooth equivariant compactifications of $\mathbb{G}_a^n$ which are Fano manifolds with index $\geq n-2$, it remains to consider Fano $n$-folds with index $\geq n-2$ and Picard number $\rho \geq 2$. The main difficulty lies on Mukai fourfolds with $\rho \geq 2$, which are classified by Wi\'sniewski in \cite{W1}. In the last section,  we will go through this list and finally classify those which can be SEC (Proposition \ref{p.Wisniewski}).
By our results in the present paper, to classify Fano 4-folds which are SEC, the only remaining case is Fano 4-folds of index 1 and with Picard number at least 2.

\vspace{5mm}

{\em Acknowledgements:} The proof of Proposition \ref{p.CC} is due to Jun-Muk Hwang, to whom we are very grateful.  We would like to thank Jaros{\l}aw Wi\'sniewski for pointing out the reference \cite{BW} to us.  We are very grateful to an anonymous referee for the detailed report which helps us to improve greatly the presentation.
Baohua Fu is supported by National Natural Science Foundation of China (No. 11431013 and 11621061). Pedro Montero is supported by National Natural Science Foundation of China (No. 11688101).

\section{Picard number one case}

Let $X$ be a uniruled projective manifold. An irreducible
component $\sK$ of the space of rational curves ${\rm RatCurves}^n(X)$ on $X$ (see \cite[Chap. II, Definition 2.11]{K}) is called {\em a minimal rational component} if the
subvariety $\sK_x$ of $\sK$ parameterizing curves passing through
a general point $x \in X$ is non-empty and proper. Curves
parameterized by $\sK$ will be called {\em minimal rational
curves}. Let $\rho: \sU \to \sK$ be the universal family and $\mu:
\sU \to X$ the evaluation map. The tangent map $\tau: \sU
\dasharrow \BP T(X)$ is defined by $\tau(u) = [T_{\mu(u)}
(\mu(\rho^{-1} \rho(u)))] \in \BP T_{\mu(u)}(X)$. The closure $\sC
\subset \BP T(X)$ of its image is the VMRT structure on $X$. The
natural projection $\sC \to X$ is a proper surjective morphism and
a general fiber $\sC_x \subset \BP T_x(X)$ is called the {\em
variety of minimal rational tangents} (VMRT for short) at the
point $x \in X$. It is well-known that $\dim \mathcal{C}_x = -K_X \cdot \ell -2$, where $\ell \in \mathcal{K}$ is a general minimal rational curve through $x\in X$.


%
\begin{examples} \label{e.IHSS}
An irreducible Hermitian symmetric space of compact type (I.H.S.S. for short) is a
homogeneous space $M= G/P$ with a simple Lie group $G$ and a
maximal parabolic subgroup $P$ such that the isotropy
representation of $P$ on $T_x(M)$ at a base point $x \in M$ is
irreducible.
 The highest weight orbit of the isotropy action on $\BP T_x(M)$
is exactly the VMRT at $x$.  The following table (e.g. \cite[Section 3.1]{FH1}) collects basic information on these varieties.
By \cite{A}, these are all SEC among rational homogeneous manifolds $G/P$ of Picard number one.
\begin{center}
\begin{tabular}{|c| c| c| c| }
\hline Type & I.H.S.S. $M$   &  VMRT  $S$ &  $S \subset \BP
T_x(M)$
\\ \hline I    &  $ \Gr(a, a+b) $ & $\pit^{a-1} \times \pit^{b-1}$
& Segre  \\  \hline II  & $\mathbb{S}_{n}$ & $\Gr(2, n)$  & Pl\"ucker  \\
\hline III & $ \Lag(2n)$ & $\pit^{n-1}$ & Veronese \\  \hline
 IV & $\Q^n$ & $\Q^{n-2}$ & Hyperquadric \\  \hline
 V & $\mathbb{O}\pit^2$ & $\mathbb{S}_{5}$ & Spinor  \\  \hline
 VI & $E_7/(E_6 \times U(1)) $ & $\mathbb{O}\pit^2$ & Severi \\
 \hline
\end{tabular}
\end{center}
\end{examples}

\begin{examples} \label{e.SymGr}
Fix two integers $k \geq 2, m\geq 1$. 
Let $\Sigma$ be an $(m+2k)$-dimensional vector space endowed with a
skew-symmetric 2-form $\omega$ of maximal rank. The symplectic Grassmannian
$M=\Gr_\omega(k, \Sigma)$ is the variety of all $k$-dimensional
isotropic subspaces of $\Sigma$, which is not homogeneous if $m$ is odd.
  Let $W$ and $Q$ be vector spaces of
dimensions $k \geq 2$ and $m$ respectively.
Let $\mathbf{t}$ be the tautological line
bundle over $\pit W$.
The VMRT $\mathcal{C}_x \subset \pit T_x(M)$ of $\Gr_\omega(k, \Sigma)$
at a general point is isomorphic to the projective
bundle $\pit((Q \otimes \mathbf{t}) \oplus \mathbf{t}^{\otimes
2})$ over $\pit W$ with the projective embedding given by
the complete linear system $$H^0(\pit W, (Q \otimes \mathbf{t}^*)
\oplus (\mathbf{t}^*)^{\otimes 2}) = (W \otimes Q)^* \oplus \SYM^2
W^*.$$  Alternatively, $\mathcal{C}_x$ is isomorphic to the blowup of $\pit^{(m+k-1)}$ along some linear space, hence it is a SEC.

\end{examples}

Recall that a subvariety $X \subset \BP V$ is called {\em conic-connected}
if through two general points there passes an irreducible conic.

\begin{proposition}[J.-M. Hwang] \label{p.CC}
Let $X \subset \BP V$ be a conic-connected smooth subvariety. Then $X$ is a SEC if and only if $X \subset \BP V$ is isomorphic to one of the following or their biregular projections:
\begin{itemize} \item[(1)] the VMRT of an irreducible Hermitian symmetric space of compact type;
\item[(2)] the VMRT of a symplectic Grassmannian $\Gr_\omega(k, 2k+m)$ for $1 \leq m, 2\leq k$;
\item[(3)] a nonsingular linear section of $\Gr(2, 5) \subset \BP^9$ of codimension $\leq 2$;
\item[(4)] a $\BP^4$-general linear section of $\mathbb{S}_5 \subset \BP^{15}$ of codimension $\leq 3$.
\end{itemize}
\end{proposition}
\begin{proof}
Assume first that $X$ is not prime Fano, then by \cite[Theorem 2.2]{IR}, $X \subset \BP V$  is projectively equivalent
to one of the following or their biregular projections:

\begin{enumerate} \item[(a1)] the second Veronese embedding of $\BP^n$;

\item[(a2)] the Segre embedding of $\BP^a \times \BP^{n-a}$ for $1
\leq a \leq n-1$.
\item[(a3)] the VMRT of the symplectic
Grassmannian $\Gr_\omega(k, 2k+m)$ for $1 \leq m, 2\leq k$.

\item[(a4)] A hyperplane section of the Segre embedding $\BP^a
\times \BP^{n+1-a}$ with $2 \leq \min\{a, n+1-a\}$.
\end{enumerate}
The case (a4) is not a SEC by the proof of \cite[Proposition 6.3]{FH1}.

Now assume that $X$ is a prime Fano manifold. Then it is an Euler symmetric variety by \cite[Corollary 5.6]{FH3}.
Let $r$ be the rank of $X$. By \cite[Theorem 3.7]{FH3}, the $r$-th fundamental form at a general point $x \in X$ is non-zero. This implies that there exists a hyperplane section $H$ such that ${\rm mult}_x(H) =r$.
Hence for any curve $C \not\subset H$ lying on $X$ through $x$, we have $H \cdot C \geq r$.
By our assumption, $X$ is conic-connected, hence the conics through $x$ cover $X$. Let $C$ be a general such conic,
then we get $2 = H \cdot C \geq r$. This implies that $r=2$, hence $X \subset \BP V$ is quadratically symmetric. Our claim follows then from \cite[Proposition 7.11, Theorem 7.13]{FH2}.
\end{proof}

Recall that for a Fano manifold $X$ of dimension $n$, its {\em index} $i_X$ is the greatest integer such that $-K_X = i_X H$ for some divisor $H$ on $X$.
 By Kobayashi-Ochiai's theorem, we have  $i_X \geq n$ if and only if $X$ is either $\pit^n$ or $\qit^n$.  A Fano manifold $X$ is called {\em del Pezzo} (resp. {\em Mukai}) if $i_X=n-1$  (resp. $i_X = n-2$).

\begin{proposition}\label{p.index}
Let $X$ be a SEC of $\rho_X=1$. Assume that the VMRT at a general point of $X$ is smooth. Then $i_X \geq 3$.
\end{proposition}
\begin{proof}
Let $D \subset X$ be the boundary divisor, which is irreducible since $\rho_X=1$. By \cite[Theorem 2.5]{HT}, we have ${\rm Pic}(X) = \mathbb{Z} D$, hence $-K_X = i_X D$. By \cite[Theorem 2.7]{HT}, we have $i_X \geq 2$.
Let $\ell$ be a minimal rational curve through a general point. Then we have $\ell \cdot D =1$ by \cite[Proposition 5.4 (v)]{FH3}. This implies that $i_X = -K_X \cdot \ell$. Assume that $i_X =2$. Then one has
$-K_X \cdot \ell =2$. Hence there exists only finitely many minimal rational curves passing through a general point, that is, the VMRT $\mathcal{C}_x$ consists of finitely many points for $x \in X$ general.
By \cite[Proposition 5.4 (ii)]{FH3},  $\mathcal{C}_x$ is irreducible and linearly non-degenerate, which yields a contradiction. This gives $i_X \geq 3$.
\end{proof}

\begin{remark}
It is expected that the assumption on the smoothness of the VMRT at a general point of $X$ in Corollary \ref{c.class} and Proposition \ref{p.index} is always fulfilled. See \cite[Conjecture 5.7]{FH3}.
\end{remark}

\begin{proposition} \label{p.DelPezzo}
Let $X$ be an $n$-dimensional del Pezzo manifold of Picard number one.  Then $X$ is a SEC if and only if $X$ is a smooth linear section of ${\rm Gr}(2,5)$ of codimension $\leq 2$.
\end{proposition}
\begin{proof}
As $\rho_X=1$,  $X$ is isomorphic to one of the following by Fujita's classification (\cite{F1},\cite{F2}, \cite{F3}),
\begin{itemize}
\item[(1)] cubic hypersurface in $\pit^{n+1}$;
\item[(2)] complete intersection of 2 quadrics in $\pit^{n+2}$;
\item[(3)] a hypersurface of degree 4 in $\pit(2,1, \ldots, 1)$;
\item[(4)]  a hypersurface of degree 6 in $\pit(3, 2,1, \ldots, 1)$;
\item[(5)] smooth linear sections of ${\rm Gr}(2,5)$;
\end{itemize}

It's well-known that for $X$ in (1) and (2), the group ${\rm Aut}^\circ (X)$ is trivial, hence $X$ is not a SEC.  In case (3), it is a double cover of $\pit^n$ ramified along a smooth hypersurface of degree 4, hence its automorphism group is finite by \cite[Theorem 4.5]{LP}.  In case (4), 
the linear system of the ample generator of ${\rm Pic}(X)$ gives a rational map $X \dasharrow \BP^{n-1}$ by \cite{F3}, which is not birational, hence
it is not a SEC by \cite[Proposition 5.4 (vi)]{FH3}. In case (5), any smooth linear section of $X$ of codimension $\leq 2$ is a SEC by \cite[Proposition 2.11]{FH2} since it is Euler symmetric. By Proposition \ref{p.index}, we have $i_X = n-1 \geq 3$, hence $n \geq 4$. This shows that the smooth linear sections of ${\rm Gr}(2,5)$ of codimension 3 and 4 are not SEC.
\end{proof}

\begin{lemma}[{\cite[Satz 8.11]{F}}] \label{l.Flenner}
Let $X$ be a smooth variety of dimension $n$, which is a complete intersection in a weighted projective space.  Then $H^p(X, \Omega_X^q(t))=0$ if $p+q < n$ and $t < q-p$.
\end{lemma}

\begin{corollary} \label{c.Mukai}
Let $X$ be a Mukai manifold of $\rho_X=1$. Assume that $X$ is a complete intersection in a weighted projective space. Then ${\rm Aut}(X)$ is discrete.
\end{corollary}
\begin{proof}
By assumption, we have $-K_X = \0_X(n-2)$.  As $T_X \simeq \Omega_X^{n-1} \otimes K_X^* = \Omega_X^{n-1}(n-2)$, we have $H^0(X, T_X)=0$ by Lemma \ref{l.Flenner}.
\end{proof}

The following lemma is well-known for {\em general} sections, but we need it for any  (smooth) section.

\begin{lemma} \label{l.QEL}
Any linear section of codimension $\leq 5$ (resp. $\leq 3$) of $\mathbb{S}_5 \subset \pit^{15}$ (resp. ${\rm Gr}(2, 6) \subset \pit^{14}$)
 is conic-connected.
\end{lemma}
\begin{proof}
By \cite[Chap. III, Proposition 2.19]{Z}, any two points of $\mathbb{S}_5 \subset \pit^{15}$ can be joined by a smooth quadric of dimension 6 contained in $\mathbb{S}_5$, which implies the first assertion. Consider the Severi variety $X={\rm Gr}(2, 6) \subset \pit^{14}$. By \cite[Chap. IV, Theorem 2.4 (b)]{Z} any two points $x, y \in X$ such that the line $\overline{xy}$ is not on $X$ are joined by a smooth quadric in $X$ of dimension 4. 
 Take any linear section $X'$ of $X$ of codimension $\leq 3$. If for two general points $x', y' \in X'$, the line $\overline{x'y'}$ is not on $X$, then 
 $x', y'$ are joined by a $\qit^4$ on $X$, and therefore they are joined by a conic on $X'$, which implies that $X'$ is conic-connected. Now assume that 
for  $x', y' \in X'$ general, the line  $\overline{x'y'}$ is on $X$, then it is also contained in $X'$ as $X'$ is a linear section of $X$. This implies that $X'$ is a projective space of dimension at least 5, which is not possible
since $X$ does not contain any $\pit^5$. This concludes the proof.
\end{proof}

\begin{proposition}\label{p.Mukai}
Let $X$ be an $n$-dimensional Mukai variety with $\rho_X=1$. Then $X$ is a SEC if and only if $X$ is one of the following
\begin{itemize}
\item[(i)] a $\pit^4$-general linear section of the 10-dimensional spinor variety $\mathbb{S}_5$ of codimension $\leq 3$.
\item[(ii)] the 8-dimensional Grassmannian ${\rm Gr}(2, 6)$.
\item[(iii)] the 6-dimensional Lagrangian Grassmannian ${\rm Lag}(6)$.
\end{itemize}
\end{proposition}
\begin{proof}
When $n=3$, then $X$ is isomorphic to $\BP^3$ or $\Q^3$ by \cite{HT}.  Now assume  $n \geq 4$. By Mukai's classification \cite{M}, $X$ is either a complete intersection in a weighted projective space, or a smooth linear section of one of the following varieties:
\begin{itemize}
\item[(a)] a quadric section of the cone over ${\rm Gr}(2,5) \subset \pit^9$;
\item[(b)] the 5-dimensional Fano contact manifold $G_2/P_2 \subset \pit^{13}$;
\item[(c)] the 6-dimensional Lagrangian Grassmannian ${\rm Lag}(6) \subset \pit^{13}$;
\item[(d)] the 10-dimensional spinor variety $\mathbb{S}_5 \subset \pit^{15}$;
\item[(e)] the 8-dimensional Grassmannian ${\rm Gr}(2, 6) \subset \pit^{14}$.
\end{itemize}

By Corollary \ref{c.Mukai}, we only need to consider  cases (a)-(e).  In case (a), the smooth linear sections of $X$ are called {\em Gushel-Mukai varieties}. By \cite[Proposition 3.19 (c)]{DK}, they have finite automorphism groups, hence they are not SEC.

For the remaining cases, $X$ is covered by lines, so it has smooth VMRT at general points. By Proposition \ref{p.index}, we have $i_X = n-2 \geq 3$, hence $n \geq 5$.

In case (b), its VMRT at a general point is linearly degenerate, hence it cannot be a SEC by \cite[Proposition 5.4 (ii)]{FH3}.

In case (c), a smooth hyperplane section $X$ of ${\rm Lag}(6)$ is a compactification of a symmetric variety with ${\rm Aut}^\circ = {\rm SL}_3$ by \cite[Theorem 3]{R}. As ${\rm SL}_3$ does not contain any subgroup isomorphic to $\mathbb{G}_a^5$,  the variety $X$ is not a SEC. Hence only
${\rm Lag}(6)$ itself is a SEC.

In cases (d) and (e), $X$ is conic-connected by Lemma \ref{l.QEL}, hence by Proposition \ref{p.CC},  $X$ is as in (i) and (ii).

\end{proof}

This concludes the proof of Theorem \ref{t.main}. Now Corollary \ref{c.class} follows by virtue of Proposition \ref{p.index}.

\section{Higher Picard number case}

\begin{proposition}\label{p.highIndex}
Let $X$ be a Fano manifold of dimension $n$ with index $i_X \geq (n+1)/2$. If $\rho_X \geq 2$, then $X$ is a SEC if and only if it is one of the following:

$$(\ast)\quad \pit^{\frac{n}{2}} \times \pit^{\frac{n}{2}}, \quad \pit^{\frac{n-1}{2}} \times \qit^{\frac{n+1}{2}}, \quad \pit_{\pit^{\frac{n+1}{2}}} (\0(1) \oplus \0^{\frac{n-1}{2}}). $$
\end{proposition}

\begin{proof}
By \cite{W2}, a Fano manifold with $i_X \geq (n+1)/2$ and $\rho_X \geq 2$ is one of the varieties in the list $(\ast)$ or the homogeneous variety $\pit T_{\pit^{\frac{n+1}{2}}}$, while the latter is not a SEC by \cite{A}. The first two varieties in $(\ast)$ are SEC. The projective bundle $\pit_{\pit^{\frac{n+1}{2}}} (\0(1) \oplus \0^{\frac{n-1}{2}})$ is isomorphic to the blowup of $\pit^n$ along a linear $\pit^{\frac{n-3}{2}}$, which is a SEC.
\end{proof}

As immediate corollaries, we have
\begin{corollary}\label{c.delPezzo}
Let $X$ be a del Pezzo manifold with $\rho_X \geq 2$, then $X$ is a SEC if and only if $X$ is one of the following:
\begin{itemize}
\item[(a)] blowup of $\pit^2$ at 1 or 2 points;
\item[(b)] $\pit^2 \times \pit^2$;
\item[(c)] blowup of $\pit^3$ at 1 point;
\item[(d)] $\pit^1 \times \pit^1 \times \pit^1$.
\end{itemize}
\end{corollary}

\begin{corollary}
Let $X$ be a Mukai manifold with $\rho_X \geq 2$.  Assume that $\dim X \geq 5$, then $X$ is a SEC if and only if $X$ is one of the following:
\begin{itemize}
\item[(a)] $\pit^3 \times \pit^3$;
\item[(b)] $\pit^2 \times \qit^3$;
\item[(c)] blowup of $\pit^4$ at a point.
\end{itemize}
\end{corollary}

Notice that the Fano SEC in dimension 3 are fully classified in \cite{HM}, while the  Mukai fourfolds are classified by Wi\'sniewski \cite{W1}. To complete the picture, it remains to determine which Fano fourfolds in the list of \cite{W1} are SEC.

\begin{proposition}\label{p.Wisniewski}
 Let $X$ be a Mukai fourfold with $\rho_X \geq 2$. Then $X$ is a SEC if and only if $X$ is one of the following:
\begin{itemize}
\item[(a)] $\pit^1\times \pit^3$;
\item[(b)] $\pit^1\times \mathbb{P}_{\mathbb{P}^2}(\mathcal{O}(1)\oplus \mathcal{O})$;
\item[(c)] $\pit^1 \times \pit^1 \times \pit^1 \times \pit^1$;
\item[(d)] blowup of $\mathbb{Q}^4$ along a line;
\item[(e)] $\mathbb{P}_{\mathbb{Q}^3}(\mathcal{O}(-1)\oplus \mathcal{O})$;
\item[(f)] $\mathbb{P}_{\mathbb{P}^3}(\mathcal{O}(-1)\oplus \mathcal{O}(1))$.
\end{itemize}
\end{proposition}

\begin{proof}
 By Wi\'sniewski's classification \cite{W1} (see \cite[Table 12.7]{IP}), $X$ is isomorphic to one of the following varieties,
 \begin{itemize}
  \item[(1)] $\mathbb{P}^1\times V$, where $V\cong \mathbb{P}T_{\mathbb{P}^2}$ or $V\cong V_d$ is a del Pezzo threefold of degree $d$ with $1\leq d \leq 5$ and $\rho_{V_d} = 1$;
  \item[(2)] $\mathbb{P}^1\times V$, where $V$ is either $\mathbb{P}^3$, or $\mathbb{P}_{\mathbb{P}^2}(\mathcal{O}(1)\oplus \mathcal{O})$ (the blowup of $\mathbb{P}^3$ at a point), or $\mathbb{P}^1\times \mathbb{P}^1 \times \mathbb{P}^1$;
  \item[(3)] a Verra fourfold, that is, a double cover of $\mathbb{P}^2\times \mathbb{P}^2$ whose branch locus is a divisor of bidegree $(2,2)$;
  \item[(4)] a divisor on $\mathbb{P}^2\times \mathbb{P}^3$ of bidegree $(1,2)$;
  \item[(5)] an intersection of two divisors of bidegree $(1,1)$ on $\mathbb{P}^3\times \mathbb{P}^3$;
  \item[(6)] a divisor on $\mathbb{P}^2\times \mathbb{Q}^3$ of bidegree $(1,1)$;
  \item[(7)] the blowup of $\mathbb{Q}^4$ along a conic which is not contained in a plane lying on $\mathbb{Q}^4$;
  \item[(8)] $\mathbb{P}_{\mathbb{P}^3}(\mathcal{N})$, where $\mathcal{N}$ is the null-correlation bundle on $\mathbb{P}^3$;
  \item[(9)] the blowup of $\mathbb{Q}^4$ along a line;
  \item[(10)] $\mathbb{P}_{\mathbb{Q}^3}(\mathcal{O}(-1)\oplus \mathcal{O})$;
  \item[(11)] $\mathbb{P}_{\mathbb{P}^3}(\mathcal{O}(-1)\oplus \mathcal{O}(1))$.
 \end{itemize}
In case (1), it follows from Blanchard's lemma \cite[Theorem 7.2.1]{B} that an effective action of $\mathbb{G}_a^4$ on $X\cong \mathbb{P}^1\times V$ descends in a unique way to an action of $\mathbb{G}_a^4$ on $V$ making the second projection an equivariant morphism. The image of the latter action is isomorphic to $\mathbb{G}_a^3$ making $V$ a SEC, while by \cite[Theorem 6.1]{HT} the only SEC threefolds with Picard number one are $\mathbb{P}^3$ and $\mathbb{Q}^3$. We conclude therefore that the variety $X$ is not a SEC. In case (2), the listed varieties are clearly SEC.

In case (3), an effective action of $\mathbb{G}_a^4$ on $X$ induces the inclusions $\mathbb{G}^4\subseteq {\rm Aut}_L(B)\subseteq {\rm Aut}(\mathbb{P}^2\times \mathbb{P}^2)$, where ${\rm Aut}_L(B)$ stands for the group of automorphisms of the branch locus $B$ induced by automorphisms of $\mathbb{P}^2\times \mathbb{P}^2$. Indeed, this follows verbatim from the proof of \cite[Lemma 4.1, Lemma 4.2, Proposition 4.3]{LP} replacing $\mathbb{P}^n$ by $\mathbb{P}^n\times \mathbb{P}^n$. In particular, we obtain in this way an effective action of $\mathbb{G}_a^4$ on $\mathbb{P}^2\times \mathbb{P}^2$ that fixes the branch locus $B$. We know on the other hand, after Hassett and Tschinkel \cite[Proposition 3.2]{HT}, that the boundary divisors for the possible effective actions of $\mathbb{G}_a^4$ on $\mathbb{P}^2\times \mathbb{P}^2$ are of bidegree $(1,0)$ and $(0,1)$, hence $X$ cannot be a SEC.

In cases (4), (5) and (6), the variety $X$ is isomorphic to the projectivization $\mathbb{P}_Y(\mathcal{B})$ of a reflexive non-locally free sheaf $\mathcal{B}$ on a smooth variety $Y$. Such sheaves are called B\v{a}nic\v{a} sheaves in \cite[Section 2]{BW}, where it is shown that the canonical map $X\to Y$ is a Mori contraction with connected but not equidimensional fibers. The general strategy to prove that none of these cases give rise to a SEC will be to analyze the points where the dimension of the fibers of $X\to Y$ jumps.

It follows from \cite[Theorem 6.8]{BW} that for $X$ as in case (4) we have $Y\cong \mathbb{P}^3$ and the canonical fibration $X\to \mathbb{P}^3$ has 8 fibers isomorphic to $\mathbb{P}^2$. By Blanchard's lemma, an effective action of $\mathbb{G}_a^4$ on $X$ induces a unique action on $\mathbb{P}^3$ for which $X\to \mathbb{P}^3$ is equivariant. As before, there is an induced effective action of $\mathbb{G}_a^3$ making $\mathbb{P}^3$ a SEC. On one hand, we notice that the 8 points $p_1,\ldots,p_8\in \mathbb{P}^3$ having 2-dimensional fibers are fixed and hence contained in the boundary hyperplane divisor $H\cong \mathbb{P}^2\subset \mathbb{P}^3$. On the other hand, if we write
$$X=\left\{\sum_{i=0}^2 x_i q_i(y_0,y_1,y_2,y_3) = 0 \right\}\subset \mathbb{P}_{\mathbf{x}}^2 \times \mathbb{P}_{\mathbf{y}}^3, $$
where the $q_i$ are quadratic forms, we have that $X\to \mathbb{P}^3$ is induced by the second projection and hence it has 2-dimensional fibers over the set $S=\{q_0(\mathbf{y})=q_1(\mathbf{y})=q_2(\mathbf{y})=0\}=\{p_1,\ldots,p_8\}$. We claim that $S$ is not contained in a hyperplane $H$ and hence $X$ is not a SEC. Indeed, if we assume that $S\subset H$ and we denote by $Q_i$ the hyperquadric $\{q_i(\mathbf{y})=0\}$, then $L_i=Q_i\cap H$ is a (possibly reducible) curve of degree 2 in $H$ and $L_{ij}=Q_i\cap Q_j$ is a curve of degree 4 in $\mathbb{P}^3$ for $i\neq j$. Since $S\subset L_{ij} \cap H$, and $S$ is a 0-dimensional scheme of length 8, it follows that $L_{ij}$ has a common component, say $N_{ij}$, contained in $H$ and thus $L_i$ and $L_j$ have a common component for $i\neq j$. Since each of the $L_i$ is a curve of degree 2, we have that if $L_0\cap L_1\cap L_2$ is 0-dimensional then each $L_i$ is reducible and given by the union of two lines in $H$. We can easily verify that $L_0\cap L_1\cap L_2$ is of length 3 and contains $S$, which is absurd. We conclude therefore that $Q_0\cap Q_1\cap Q_3$ have a common component. However, this is not possible as $X$ is irreducible. The cases (5) and (6) are similar but easier: in the former case we have $Y\cong \mathbb{P}^3$ and the canonical fibration $X\to \mathbb{P}^3$ has 4 fibers isomorphic to $\mathbb{P}^2$ by \cite[Theorem 6.8]{BW}. More precisely, for each $p_1,\ldots,p_4\in \mathbb{P}^3$ having 2-dimensional fiber, the fiber is given by the dual hyperplane $H_{p_i}$ in $(\mathbb{P}^3)^\vee$ determined by $p_i$. Since each of the points $p_i$ are fixed we deduce that each $H_{p_i}$ is invariant under the induced action of $\mathbb{G}_a^3$ in $(\mathbb{P}^3)^\vee$. Thus, we get four different invariant divisors, a contradiction. In the latter case we have $Y\cong \mathbb{Q}^3$, and the canonical fibration $X\to \mathbb{Q}^3$ has 2 fibers isomorphic to $\mathbb{P}^2$. Hence the result follows from the fact \cite[Theorem 6.1]{HT} that there is a unique effective action of $\mathbb{G}_a^3$ making $\mathbb{Q}^3$ a SEC with a unique fixed point.

In cases (7) and (9) the variety $X$ is isomorphic to the blowup of $\mathbb{Q}^4$ along a smooth curve $C\subset \mathbb{Q}^4$. By Blanchard's lemma, $X$ is a SEC if and only if $C$ is invariant under the unique effective action of $\mathbb{G}_a^4$ on $\mathbb{Q}^4$. A simple computation in coordinates shows that the only invariant smooth curves on $\mathbb{Q}^4$ are lines. Thus in case (7) $X$ is not a SEC, while in case (9) it is.

In case (8), it follows from \cite[Theorem 3.1]{CP} that $X\cong \mathbb{P}_{\mathbb{P}^3}(\mathcal{N})$ is isomorphic a homogeneous space $G/P$. Hence it follows from \cite{A} that $X$ is not a SEC.

In cases (10) and (11) there is a blowdown $X\to Z$ sending the divisor corresponding to a section of the $\pit^1$-bundle structure of $X$ to a point $z\in Z$, where $Z\cong \mathbb{Q}_0^4$ is the cone over $\mathbb{Q}^3$ in $\mathbb{P}^5$ and $Z\cong \mathbb{P}(1,1,1,1,2)$, respectively, and $z\in Z$ is the only singular point of each of these varieties. In both cases $Z$ is a SEC (cf. \cite[Section 6]{AP}, \cite[Proposition 2]{AR}) and $z\in Z$ is a fixed point under the respective actions since is the only singular point. We conclude therefore that in both cases (10) and (11) $X$ is a SEC.
\end{proof}

\bigskip
Baohua Fu (bhfu@math.ac.cn)

MCM, AMSS, Chinese Academy of Sciences, 55 ZhongGuanCun East Road, Beijing, 100190, China
and
 School of Mathematical Sciences, University of Chinese Academy of Sciences, Beijing, China

 \bigskip

 Pedro Montero (pmontero@amss.ac.cn)

 AMSS, Chinese Academy of Sciences, 55 ZhongGuanCun East Road, Beijing, 100190, China

\end{document}